\documentclass{amsart}
\usepackage{latexsym,amsxtra,amscd,ifthen}
\usepackage{amsfonts}
\usepackage{verbatim}
\usepackage{amsmath}
\usepackage{amsthm}
\usepackage{amssymb}

\allowdisplaybreaks

\numberwithin{equation}{section}

\theoremstyle{plain}
\newtheorem{theorem}{Theorem}[section]
\newtheorem{lemma}[theorem]{Lemma}
\newtheorem{proposition}[theorem]{Proposition}

\theoremstyle{definition}

\makeatletter              % This sequence of commands will
\let\c@equation\c@theorem  % incorporate equation numbering
\makeatother

\begin{document}

\title[]
{Weak Hopf algebra structures on hybrid numbers}

\author{Tang Jiangang$^{1}$ and  Chen Quanguo$^{2}$   }
\address{$^{ 1 }$Division of Mathematics, Sichuan University  Jinjiang College,  Sichuan, Meishan, China; $^{ 2 }$College of Mathematics and Statistics, Kashi University, Xinjiang, Kashi, China\\}
\email{ cqg211@163.com (Q.G.Chen$^{\ast}$   Corresponding author), jg-tang@163.com(J. G. Tang)}
\thanks{This work was supported by the National Natural Science
	Foundation of China (No. 12271292) and the Natural Foundation of Shandong Province(No. ZR2022MA002).}

\begin{abstract}
Let $\mathbb{K}$ be the set of hybrid numbers. This paper is to look for all the weak Hopf structures  on $\mathbb{K}$. Once $\mathbb{K}$ is endowed with a structure of a weak Hopf algebra, we shall compute the source algebra and target algebra of $\mathbb{K}$. Also, we shall discuss the integrals in such weak Hopf algebra  $\mathbb{K}$.
\end{abstract}


%16A62 (1973-1990) Homological methods

%16E70 (1991-1999) Other rings of low
%           global or flat dimension

%16W30 (1991-now) Coalgebras, bialgebras,
%      Hopf algebras [See also 16S40, 57T05];

%20J05 (1973-now) Homological methods in group theory


%\thanks{This work was partially supported by a grant
% from the Simons Foundation (\#208314 to Ellen Kirkman).}
%\thanks{the US National Science Foundation.}

\maketitle

\setcounter{section}{0}
\section{Introduction}
\label{xxsec0}
Weak Hopf algebras were introduced by B\"{o}hm, Nill and Szlach\'{a}nyi in \cite{Bohm} as a generalization of ordinary Hopf algebras and groupoid algebras. In the notion of a weak Hopf algebra,  the conditions which $\Delta$, $\varepsilon$ and the antipode $S$ satisfy are weakened, for example, the comultiplication of a weak Hopf algebra is no longer required to preserve the unit. The axioms of a weak Hopf algebra are self-dual, which ensures that, for a finite dimensional weak Hopf algebra $H$, its dual vector space $H^{\ast}$ has a natural structure of a weak Hopf algebra. The motivation to study weak Hopf algebras was their connection with the theory of algebra extensions \cite{Kadison} and their important application that they provide a natural framework for the study of dynamical twists in Hopf algebra \cite{Etingof}.

As a generalization of complex, hyperbolic and dual numbers,  a new non-commu
tative number system called hybrid numbers are introduced in \cite{ozdemir}. A hybrid number is a number created with any combination of the complex, hyperbolic and dual numbers satisfying
the certain relation.

The purpose of this paper is to discuss  all the weak Hopf algebra structures on  $\mathbb{K}$ (the set of hybrid numbers).

The paper is organized as follows. In Sec. 2, we recall basic definitions and give some fundamental properties concerning weak Hopf algebras and hybrid numbers. In Sec. 3, we shall discuss all the weak Hopf algebra structures on $\mathbb{K}$, and compute the source algebra and target algebra of such weak Hopf algebra. In Sec. 4, we shall discuss the left and right integrals in such weak Hopf algebra.
\section{Preliminaries}
\label{xxsec0}
Throughout the paper, all vector spaces, linear maps and tensor products are over $\mathbb{C}$.
We follow Montgomery’s book for terminologies on algebras, coalgebras and modules, but omit the usual summation
indices and summation symbols.

\subsection{ Weak Hopf algebras}
A sixtuple $(H, m, 1, \Delta, \varepsilon, S)$ is a weak Hopf algebra with antipode $S$, if $(H,m,1)$ is an algebra and $(H, \Delta, \varepsilon)$ is a coalgebra, and the following conditions hold:
\begin{equation}\label{eq1}
	\Delta(kh)=\Delta(k)\Delta(h),  \forall  h,k\in H,
\end{equation}
\begin{equation}\label{eq2}
\varepsilon(kh_{(1)})\varepsilon(h_{(2)}g)=\varepsilon(khg)=\varepsilon(kh_{(2)})\varepsilon(h_{(1)}g), \forall h,k,g\in H,
\end{equation}
\begin{equation}\label{eq3}
	(1\otimes \Delta(1))( \Delta(1)\otimes1)=\Delta^{2}(1)=(\Delta(1)\otimes 1)( 1\otimes\Delta(1)),
\end{equation}
\begin{equation}\label{eq4}
h_{(1)}S(h_{(2)})=\varepsilon_{t}(h), S(h_{(1)})h_{(2)}=\varepsilon_{s}(h), S(h_{(1)})h_{(2)}S(h_{(3)})=S(h),
\end{equation}
where $\varepsilon_{t}, \varepsilon_{s}: H\rightarrow H $ are defined by $\varepsilon_{t}(h)=\varepsilon(1_{(1)}h)1_{(2)}, \varepsilon_{s}(h)=\varepsilon(h1_{(2)})1_{(1)}$. We will denote $H_{t}=\varepsilon_{t}(H)$ and $H_{s}=\varepsilon_{s}(H)$.

Recall from \cite{Pierce} that an algebra $A$ is separable if and only if there exists a $q\in A\otimes A$ such
that
$$
(x\otimes 1)q=q(1\otimes x)
$$
holds, for all $x\in A$ and furthermore $m_{A}(q) = 1$. Such a $q$ will be called a separable idempotent.  Let $H$ be a weak Hopf algebra. By Proposition 2.11 in \cite{Bohm}, we know that $H_{t}(H_{s})$ is a separable algebra with the separable idempotent given by
$$
q=S(1_{(1)})\otimes 1_{(2)}(q=1_{(1)}\otimes S( 1_{(2)})).
$$

A left (right) integral in a weak Hopf algebra $H$ is an element $l\in H (r\in H)$
satisfying
$$
xl=\varepsilon_{t}(x)l \quad (rx=r\varepsilon_{s}(x)),
$$
for all $x\in H$. The space of left (right) integrals in $H$ is denoted by $I^{L}(H)$ $(I^{R}(H))$.

\subsection{Hybrid numbers} The set of hybrid numbers, denoted by $\mathbb{K}$, is defined as
$$
\mathbb{K}=\{a+b g+c\mu+d\nu|a, b,c, d\in \mathbb{C}, g^{2}=-1, \mu^{2}=0,\nu^{2}=1, g\nu =-\nu g =\mu +g\}.
$$
Using these equalities $$ g^{2}=-1, \mu^{2}=0,\nu^{2}=1, g\nu =-\nu g =\mu +g,$$ we can gain the following multiplication table $$\begin{tabular}{|c|c|c|c|c|}
	\hline
	& 1 & $g$ & $\mu$ & $\nu$ \\
	\hline
	1	& 1 & $g$ & $\mu$ & $\nu$ \\
	\hline
	$g$	& $g$ & -1 & 1-$\nu$ & $\mu$+$g$ \\
	\hline
	$\mu$	& $\mu$ & $\nu$+1 & 0 &  -$\mu$\\
	\hline
	$\nu$	& $\nu$ &-$\mu$-$g$  & $\mu$ &  1\\
	\hline
\end{tabular}$$
This table shows us that the multiplication operation in the
hybrid numbers is not commutative. But it has the property of associativity.
\section{Weak Hopf algebra structures on $\mathbb{K}$}
\begin{lemma}\label{lem1}
	$\mathbb{K}$ is endowed with the coalgebra structure as
	$$
\begin{cases}
\Delta(1)=\frac{1}{2} 1\otimes 1-\frac{1}{2} \mu\otimes \mu-\frac{i}{2} \nu\otimes \mu-\frac{i}{2} \mu\otimes \nu+\frac{1}{2} \nu\otimes \nu, & \\
		\Delta(g)=-\frac{1}{2} \mu\otimes 1-\frac{i}{2} \nu\otimes 1+b g\otimes g
		+b \mu \otimes g + ib \nu\otimes g
		-\frac{1}{2} 1\otimes \mu+b g\otimes \mu
		+b \mu \otimes \mu \\
		\hspace{1cm}+ ib \nu\otimes \mu
		-\frac{i}{2} 1\otimes \nu+ib g\otimes \nu
		+ib \mu \otimes \nu -b \nu\otimes \nu, & \\
	\Delta(\mu)=\frac{1}{2b}\mu\otimes\mu,&\\
	\Delta(\nu)=-\frac{i}{2}\mu \otimes 1+\frac{1}{2}\nu\otimes 1
	-\frac{i}{2}1 \otimes \mu+\frac{i}{2b}\mu\otimes \mu+\frac{1}{2}1\otimes \nu,&\\
	\varepsilon(1)=2, \varepsilon(g)=\frac{1}{b},\varepsilon(\mu)=2b, \varepsilon(\nu)=2ib,
\end{cases}
$$
where $b$ is a non-zero parameter. 	
\end{lemma}	
\begin{proof}
Firstly, we shall check that, for all $x\in \{1, g,\mu, \nu\}$, the following equality holds,
$$
(\Delta \otimes id)\Delta(x)=(id\otimes \Delta)\Delta(x).
$$
For $x=\nu$, on one hand,
\begin{eqnarray*}
(\Delta \otimes id)\Delta(\nu)
&=&(\Delta \otimes id)(-\frac{i}{2}\mu \otimes 1+\frac{1}{2}\nu\otimes 1
-\frac{i}{2}1 \otimes \mu+\frac{i}{2b}\mu\otimes \mu+\frac{1}{2}1\otimes \nu)\\
&=&-\frac{i}{4}\mu\otimes 1\otimes 1+\frac{1}{4}\nu\otimes 1\otimes 1-\frac{i}{4}1\otimes \mu\otimes 1+\frac{1}{4}1\otimes \nu\otimes 1\\
&&-\frac{i}{4}1\otimes 1\otimes \mu+\frac{ib^{2}+i}{4b^{2}}\mu\otimes\mu\otimes\mu-\frac{1}{4}\nu\otimes \mu\otimes \mu\\
&&-\frac{1}{4}\mu\otimes \nu\otimes \mu-\frac{i}{4}\nu\otimes \nu\otimes \mu+\frac{1}{4}1\otimes 1\otimes \nu\\
&&-\frac{1}{4}\mu\otimes \mu\otimes \nu-\frac{i}{4}\nu\otimes \mu\otimes \nu-\frac{i}{4}\mu\otimes \nu\otimes \nu+\frac{1}{4}\nu\otimes \nu\otimes \nu,
\end{eqnarray*}
on the other hand,
\begin{eqnarray*}
	(id\otimes \Delta)\Delta(\nu)
	&=&(id\otimes \Delta)(-\frac{i}{2}\mu \otimes 1+\frac{1}{2}\nu\otimes 1
	-\frac{i}{2}1 \otimes \mu+\frac{i}{2b}\mu\otimes \mu+\frac{1}{2}1\otimes \nu)\\
	&=&-\frac{i}{4}\mu\otimes 1\otimes 1+\frac{ib^{2}+i}{4b^{2}}\mu\otimes \mu\otimes \mu-\frac{1}{4}\mu\otimes \nu\otimes \mu-\frac{1}{4}\mu\otimes \mu\otimes \nu\\
	&&-\frac{i}{4}\mu\otimes \nu\otimes \nu+\frac{1}{4}\nu\otimes1\otimes1-\frac{1}{4}\nu\otimes \mu\otimes \mu\\
	&&-\frac{i}{4}\nu\otimes \nu\otimes \mu-\frac{i}{4}\nu\otimes \mu\otimes \nu+\frac{1}{4}\nu\otimes \nu\otimes \nu\\
	&&-\frac{i}{4}1\otimes \mu\otimes 1+\frac{1}{4}1\otimes \nu\otimes 1-\frac{i}{4}1\otimes 1\otimes \mu+\frac{1}{4}1\otimes 1\otimes \nu,
\end{eqnarray*}
by comparing the two results, we can get $$
(\Delta \otimes id)\Delta(\nu)=(id\otimes \Delta)\Delta(\nu).
$$
Since
\begin{eqnarray*}
	&&(id\otimes \varepsilon) \Delta(g)\\
	&=&(id\otimes \varepsilon)(-\frac{1}{2} \mu\otimes 1-\frac{i}{2} \nu\otimes 1+b g\otimes g
	+b \mu \otimes g + ib \nu\otimes g\\
&&	-\frac{1}{2} 1\otimes \mu+b g\otimes \mu
	+b \mu \otimes \mu+ ib \nu\otimes \mu	-\frac{i}{2} 1\otimes \nu\\
&&
+ib g\otimes \nu
	+ib \mu \otimes \nu -b \nu\otimes \nu)\\
&=&-\mu-i\nu +g+\mu+i\nu -b1+2b^{2}g+2b^{2}\mu+2ib^{2}\nu+b1-2b^{2}g-2b^{2}\mu-2ib^{2}\nu\\
&=&g
\end{eqnarray*}
and
\begin{eqnarray*}
	&&(\varepsilon\otimes id) \Delta(g)\\
	&=&(\varepsilon\otimes id) (-\frac{1}{2} \mu\otimes 1-\frac{i}{2} \nu\otimes 1+b g\otimes g
	+b \mu \otimes g + ib \nu\otimes g\\
	&&	-\frac{1}{2} 1\otimes \mu+b g\otimes \mu
	+b \mu \otimes \mu+ ib \nu\otimes \mu	-\frac{i}{2} 1\otimes \nu\\
	&&
	+ib g\otimes \nu
	+ib \mu \otimes \nu -b \nu\otimes \nu)\\
	&=&-b1+b1+g+2b^{2}g-2b^{2}g-\mu+\mu+2b^{2}\mu-2b^{2}\mu-i\nu+i\nu +2ib^{2}\nu-2ib^{2}\nu\\
	&=&g,
\end{eqnarray*}
it follows that $$(id\otimes \varepsilon) \Delta(g)=g=(\varepsilon\otimes id) \Delta(g).$$
Similarly, we have
$$
(id\otimes \varepsilon) \Delta(x)=x=(\varepsilon\otimes id) \Delta(x), x\in \{1,\mu,\nu\}.
$$
So one has
$$
(id\otimes \varepsilon) \Delta=id=(\varepsilon\otimes id) \Delta.
$$
\end{proof}	
\begin{lemma}\label{lem2}
	With the comultiplication $\Delta$   given in Lemma \ref{lem1}. Then $\Delta$ satisfies the conditions (\ref{eq1}) and (\ref{eq3}).
\end{lemma}
\begin{proof}
	Since
	\begin{eqnarray*}
		&&	(1\otimes \Delta(1))( \Delta(1)\otimes1)\\
		&=&(1\otimes (\frac{1}{2} 1\otimes 1-\frac{1}{2} \mu\otimes \mu-\frac{i}{2} \nu\otimes \mu-\frac{i}{2} \mu\otimes \nu+\frac{1}{2} \nu\otimes \nu))\\
		&&((\frac{1}{2} 1\otimes 1-\frac{1}{2} \mu\otimes \mu-\frac{i}{2} \nu\otimes \mu-\frac{i}{2} \mu\otimes \nu+\frac{1}{2} \nu\otimes \nu)\otimes 1)\\
	&=& \frac{1}{4}1\otimes 1\otimes 1-\frac{1}{4}\mu\otimes \mu\otimes 1-\frac{i}{4}\nu\otimes \mu\otimes 1-\frac{i}{4}\mu\otimes \nu\otimes 1+\frac{1}{4}\nu\otimes \nu\otimes 1\\
	&& -\frac{1}{4}1\otimes \mu\otimes \mu-\frac{i}{4}1\otimes \nu\otimes \mu-\frac{i}{4}1\otimes \mu\otimes \nu+\frac{1}{4}1\otimes \nu\otimes \nu\\
		&& -\frac{1}{4}\mu\otimes 1\otimes \mu-\frac{i}{4}\nu\otimes 1\otimes \mu-\frac{i}{4}\mu\otimes 1\otimes \nu+\frac{1}{4}\nu\otimes 1\otimes \nu
	\end{eqnarray*}
and
\begin{eqnarray*}
	\Delta^{2}(1)
	&=&(id\otimes \Delta)(\frac{1}{2} 1\otimes 1-\frac{1}{2} \mu\otimes \mu-\frac{i}{2} \nu\otimes \mu-\frac{i}{2} \mu\otimes \nu+\frac{1}{2} \nu\otimes \nu)\\
	&=&\frac{1}{4}1\otimes 1\otimes 1-\frac{1}{4}1\otimes \mu\otimes \mu-\frac{i}{4}1\otimes \nu\otimes \mu-\frac{i}{4}1\otimes \mu\otimes \nu\\
	&&+\frac{1}{4}1\otimes \nu\otimes \nu-\frac{1}{4}\mu\otimes \mu\otimes 1-\frac{i}{4}\mu\otimes \nu\otimes 1\\
&&-	\frac{1}{4}\mu\otimes 1\otimes \mu-	\frac{i}{4}\mu\otimes 1\otimes \nu-	\frac{i}{4}\nu\otimes \mu\otimes 1+	\frac{1}{4}\nu\otimes \nu\otimes 1\\
&&-	\frac{i}{4}\nu\otimes 1\otimes \mu+	\frac{1}{4}\nu\otimes 1\otimes \nu,
\end{eqnarray*}
it follows that
$$
(1\otimes \Delta(1))( \Delta(1)\otimes1)=	\Delta^{2}(1).
$$
Similarly, we can check that
$$
\Delta^{2}(1)=(\Delta(1)\otimes 1)( 1\otimes\Delta(1))
$$
holds.

Now, we shall check that $\Delta$ preserves the multiplication. For example,
\begin{eqnarray*}
	&&\Delta(g)\Delta(\nu)\\
	&=&(-\frac{1}{2} \mu\otimes 1-\frac{i}{2} \nu\otimes 1+b g\otimes g
	+b \mu \otimes g + ib \nu\otimes g\\
	&&-\frac{1}{2} 1\otimes \mu+b g\otimes \mu
	+b \mu \otimes \mu + ib \nu\otimes \mu\\
	&&-\frac{i}{2} 1\otimes \nu+ib g\otimes \nu
	+ib \mu \otimes \nu -b \nu\otimes \nu)(-\frac{i}{2}\mu \otimes 1+\frac{1}{2}\nu\otimes 1\\
	&&-\frac{i}{2}1 \otimes \mu+\frac{i}{2b}\mu\otimes \mu+\frac{1}{2}1\otimes \nu)\\
	&=&-\frac{1}{4}\mu\nu\otimes 1+\frac{i}{2}\mu\otimes\mu-\frac{1}{2}\mu\otimes\nu-\frac{1}{4}\nu\mu\otimes1
	-\frac{i}{4}\nu^{2}\otimes1\\
	&&-\frac{1}{2}\nu\otimes\mu+\frac{1+2b^{2}}{4b}\nu\mu\otimes\mu-\frac{i}{2}\nu\otimes\nu-\frac{bi}{2}g\mu\otimes g
	+\frac{b}{2}g\nu\otimes g\\
	&&-\frac{ib}{2}g\otimes g\mu+\frac{i}{2}g\mu\otimes g\mu+\frac{b}{2}g \otimes g\nu+\frac{b}{2}\mu\nu\otimes g-\frac{ib}{2}\mu\otimes g\mu\\
	&&+\frac{b}{2}\mu\otimes g\nu+\frac{b}{2}\nu\mu\otimes g+\frac{ib}{2}\nu^{2} \otimes g+\frac{b}{2}\nu\otimes g\mu-\frac{1}{2}\nu\mu\otimes g\mu\\
	&&+\frac{ib}{2}\nu\otimes g\nu-\frac{1}{4}1\otimes \mu\nu-\frac{ib}{2}g\mu \otimes \mu+\frac{b}{2}g\nu\otimes \mu+\frac{b}{2}g\otimes \mu\nu\\
	&&+\frac{b}{2}\mu\nu\otimes \mu+\frac{b}{2}\mu\otimes \mu\nu+\frac{ib}{2}\nu^{2} \otimes \mu+\frac{ib}{2}\nu\otimes \mu\nu\\
		&&-\frac{1}{4}1\otimes \nu\mu+\frac{1}{4  b}\mu\otimes \nu\mu-\frac{i}{4}1 \otimes \nu^{2}+\frac{b}{2}g\mu\otimes \nu+\frac{ib}{2}g\nu\otimes \nu\\
	&&+\frac{b}{2}g\otimes \nu\mu-\frac{1}{2}g\mu\otimes \nu\mu+\frac{ib}{2}g\otimes \nu^{2}+\frac{ib}{2}\mu\nu\otimes \nu+\frac{b}{2}\mu\otimes \nu\mu+\frac{ib}{2}\mu\otimes \nu^{2}\\
	&&+\frac{bi}{2}\nu\mu\otimes \nu-\frac{b}{2}\nu^{2}\otimes \nu+\frac{ib}{2}\nu\otimes \nu\mu-\frac{i}{2}\nu\mu\otimes \nu\mu-\frac{b}{2}\nu\otimes \nu^{2}\\
&=&-\frac{1}{2} \mu\otimes 1-\frac{i}{2} \nu\otimes 1+b g\otimes g
+b \mu \otimes g + ib \nu\otimes g
-\frac{1}{2} 1\otimes \mu+b g\otimes \mu\\
&&
+(b+\frac{1}{2b}) \mu \otimes \mu + ib \nu\otimes \mu
-\frac{i}{2} 1\otimes \nu+ib g\otimes \nu
+ib \mu \otimes \nu -b \nu\otimes \nu\\
&=&\Delta(\mu+g)=\Delta(g\nu).
\end{eqnarray*}
\end{proof}
\begin{lemma}\label{lem3}
	With the counit $\varepsilon$  given in Lemma \ref{lem1}. Then $\varepsilon$ satisfies the conditions (\ref{eq2}).
\end{lemma}

\begin{proof}
	Since
	\begin{eqnarray*}
		\varepsilon(\mu g_{(1)})\varepsilon(g_{(2)}\nu)
		&=&-\frac{1}{2}\varepsilon(\mu \mu)\varepsilon(\nu)-\frac{i}{2}\varepsilon(\mu \nu)\varepsilon(\nu)
		+b\varepsilon(\mu g)\varepsilon(g\nu)\\
		&&+b\varepsilon(\mu \mu)\varepsilon(g\nu)+ib\varepsilon(\mu \nu)\varepsilon(g\nu)-\frac{1}{2}\varepsilon(\mu )\varepsilon(\mu\nu)\\
		&&+b\varepsilon(\mu g)\varepsilon(\mu\nu)+b\varepsilon(\mu \mu)\varepsilon(\mu\nu)+ib\varepsilon(\mu\nu )\varepsilon(\mu\nu)\\
		&&-\frac{i}{2}\varepsilon(\mu )\varepsilon(\nu\nu)+ib\varepsilon(\mu g)\varepsilon(\nu\nu)+ib\varepsilon(\mu g )\varepsilon(\nu\nu)\\
		&&+ib\varepsilon(\mu \mu)\varepsilon(\nu\nu)-b\varepsilon(\mu \nu)\varepsilon(\nu\nu)\\
&=&\frac{i}{2}\varepsilon(\mu )\varepsilon(\nu)
		+b\varepsilon(\nu+1)\varepsilon(\mu+g)+ib\varepsilon(- \mu)\varepsilon(\mu+g)\\
		&&-\frac{1}{2}\varepsilon(\mu )\varepsilon(-\mu)+b\varepsilon(\nu+1)\varepsilon(-\mu)+ib\varepsilon(-\mu )\varepsilon(-\mu)\\
		&&-\frac{i}{2}\varepsilon(\mu )\varepsilon(1)+ib\varepsilon(\nu+1)\varepsilon(1)+ib\varepsilon(\nu+1 )\varepsilon(1)\\
		&&-b\varepsilon(-\mu )\varepsilon(1)\\
		&=&-2b^{2}+b(2+2ib)(2b+\frac{1}{b})+ib(-2b)(2b+\frac{1}{b})\\
		&&-\frac{1}{2}2b(-2b)+b(2ib+2)(-2b)+ib(-2b)(-2b)\\
		&&-\frac{i}{2} 4b+ib(2ib+2)2+ib(2ib+2 )2+4b^{2}\\
	&=&2ib+2\\
	&=&\varepsilon(\mu g\nu)
	\end{eqnarray*}
	and
	\begin{eqnarray*}
		\varepsilon(\mu g_{(2)})\varepsilon(g_{(1)}\nu)
		&=&-\frac{1}{2}\varepsilon(\mu )\varepsilon(\mu\nu)-\frac{i}{2}\varepsilon(\mu )\varepsilon(\nu\nu)
		+b\varepsilon(\mu g)\varepsilon(g\nu)\\
		&&+b\varepsilon(\mu g)\varepsilon(\mu\nu)+ib\varepsilon(\mu g)\varepsilon(\nu\nu)-\frac{1}{2}\varepsilon(\mu\mu )\varepsilon(\nu)\\
		&&+b\varepsilon(\mu \mu)\varepsilon(g\nu)+b\varepsilon(\mu \mu)\varepsilon(\mu\nu)+ib\varepsilon(\mu\mu )\varepsilon(\nu\nu)\\
		&&-\frac{i}{2}\varepsilon(\mu \nu)\varepsilon(\nu)+ib\varepsilon(\mu \nu)\varepsilon(g\nu)+ib\varepsilon(\mu \nu )\varepsilon(g\nu)\\
		&&+ib\varepsilon(\mu \nu)\varepsilon(\mu\nu)-b\varepsilon(\mu \nu)\varepsilon(\nu\nu)\\
		&=&-\frac{1}{2}\varepsilon(\mu )\varepsilon(\mu\nu)-\frac{i}{2}\varepsilon(\mu )\varepsilon(\nu\nu)
	+b\varepsilon(\mu g)\varepsilon(g\nu)\\
	&&+b\varepsilon(\mu g)\varepsilon(\mu\nu)+ib\varepsilon(\mu g)\varepsilon(\nu\nu)+ib\varepsilon(\mu\mu )\varepsilon(\nu\nu)\\
	&&-\frac{i}{2}\varepsilon(\mu \nu)\varepsilon(\nu)+ib\varepsilon(\mu \nu)\varepsilon(g\nu)+ib\varepsilon(\mu \nu )\varepsilon(g\nu)\\
	&&+ib\varepsilon(\mu \nu)\varepsilon(\mu\nu)-b\varepsilon(\mu \nu)\varepsilon(\nu\nu)\\
		&=&-\frac{1}{2}\varepsilon(\mu )\varepsilon(-\mu)-\frac{i}{2}\varepsilon(\mu )\varepsilon(1)
	+b\varepsilon(\nu+1)\varepsilon(\mu+g)\\
	&&+b\varepsilon(\nu+1)\varepsilon(-\mu)+ib\varepsilon(\nu+1)\varepsilon(1)\\
	&&-\frac{i}{2}\varepsilon(-\mu)\varepsilon(\nu)+ib\varepsilon(-\mu )\varepsilon(\mu+g)+ib\varepsilon(-\mu )\varepsilon(\mu+g)\\
	&&+ib\varepsilon(-\mu)\varepsilon(-\mu)-b\varepsilon(-\mu)\varepsilon(1)\\
		&=&2b^{2}-2ib
	+b(2+2ib)(2b+\frac{1}{b})+b(2+2ib)(-2b)+2ib(2+2ib)\\
	&&\frac{i}{2}2b2ib-ib2b(2b+\frac{1}{b})-ib2b(2b+\frac{1}{b})+ib2b2b+4b^{2}\\
	&=&2ib+2\\
	&=&\varepsilon(\mu g\nu),
		\end{eqnarray*}
		it follows that
		$$
		\varepsilon(\mu g_{(2)})\varepsilon(g_{(1)}\nu)=\varepsilon(\mu g\nu)=	\varepsilon(\mu g_{(1)})\varepsilon(g_{(2)}\nu).
		$$
Similarly, the condition (\ref{eq2}) holds for the others.
\end{proof}

To sum up, we have the main result in this section.
\begin{theorem}\label{thm1}
		With the coalgebra structure on  $\mathbb{K}$   given in Lemma \ref{lem1}.  $\mathbb{K}$ is a weak Hopf algebra with the
		antipode given by
		$$
		S(1)=1,S(g)=(-1+\frac{1}{2b^{2}})\mu-i\nu,
		$$
		$$
		S(\mu)=2b^{2}g+2b^{2}\mu+2i b^{2}\nu,
		$$	
		$$
		S(\nu)=2ib^{2}g-(i-2ib^{2})\mu+ (1-2b^{2})\nu.
		$$	
\end{theorem}
\begin{proof}
	By Lemmas \ref{lem1},\ref{lem2} and \ref{lem3}, we know that  $\mathbb{K}$ is a bialgebra. Next, we shall check that $S$ satisfies the condition (\ref{eq4}). For example,
	\begin{eqnarray*}
		\nu_{(1)}S(\nu_{(2)})
		&=& -\frac{i}{2}\mu S(1)+\frac{1}{2}\nu S(1)-\frac{i}{2} S(\mu)+\frac{i}{2b}\mu S(\mu)+\frac{1}{2} S(\nu)\\
		&=& -\frac{i}{2}\mu +\frac{1}{2}\nu -\frac{i}{2} (2b^{2}g+2b^{2}\mu+2i b^{2}\nu)\\
		&&+\frac{i}{2b}\mu (2b^{2}g+2b^{2}\mu+2i b^{2}\nu)+\frac{1}{2} (2ib^{2}g-(i-2ib^{2})\mu+ (1-2b^{2})\nu)\\
		&=&(b-i)\mu+(1+ib)\nu+ib
			\end{eqnarray*}
	and
	\begin{eqnarray*}
	\varepsilon(1_{(1)}\nu)1_{(2)}
	&=&\frac{1}{2}\varepsilon(\nu)1-\frac{1}{2}\varepsilon(\mu\nu)\mu-\frac{i}{2}\varepsilon(\nu\nu)\mu
	-\frac{i}{2}\varepsilon(\mu\nu)\nu+\frac{1}{2}\varepsilon(\nu\nu)\nu\\
		&=&\frac{1}{2}\varepsilon(\nu)1-\frac{1}{2}\varepsilon(-\mu)\mu-\frac{i}{2}\varepsilon(1)\mu
	-\frac{i}{2}\varepsilon(-\mu)\nu+\frac{1}{2}\varepsilon(1)\nu\\
	&=&(b-i)\mu+(1+ib)\nu+ib,
	\end{eqnarray*}
so	this shows that
	$$
		\nu_{(1)}S(\nu_{(2)})=\varepsilon(1_{(1)}\nu)1_{(2)}.
	$$
Since 		\begin{eqnarray*}
		S(\nu_{(1)})\nu_{(2)}
		&=& -\frac{i}{2}S(\mu) +\frac{1}{2}S(\nu)-\frac{i}{2} \mu+\frac{i}{2b}S(\mu)\mu +\frac{1}{2} \nu\\
		&=& -\frac{i}{2}(2b^{2}g+2b^{2}\mu+2i b^{2}\nu) +\frac{1}{2}(2ib^{2}g-(i-2ib^{2})\mu+ (1-2b^{2})\nu)\\
		&&-\frac{i}{2} \mu+\frac{i}{2b}(2b^{2}g+2b^{2}\mu+2i b^{2}\nu)\mu +\frac{1}{2} \nu\\
		&=&bi+(1-bi)\nu-(b+i)\mu
	\end{eqnarray*}
and
	\begin{eqnarray*}
1_{(1)}\varepsilon(\nu 1_{(2)})	&=&\frac{1}{2}1	\varepsilon(\nu 1)-\frac{1}{2}\mu	\varepsilon(\nu \mu)-\frac{i}{2}\nu	\varepsilon(\nu \mu)
	-\frac{i}{2}\mu\varepsilon(\nu \nu)+\frac{1}{2}\varepsilon(\nu\nu)\nu\\
	&=&\frac{1}{2}1	\varepsilon(\nu 1)-\frac{1}{2}\mu	\varepsilon( \mu)-\frac{i}{2}\nu	\varepsilon(\mu)
	-\frac{i}{2}\mu\varepsilon(1)+\frac{1}{2}\varepsilon(1)\nu\\
	&=&bi+(1-bi)\nu-(b+i)\mu,
\end{eqnarray*}
it follows that
$$
S(\nu_{(1)})\nu_{(2)}=1_{(1)}\varepsilon(\nu 1_{(2)}).
$$
The others can be checked similarly.
\end{proof}

Except for the weak Hopf algebra structure on $\mathbb{K}$ given in Theorem \ref{thm1}, there exists another weak Hopf algebra structure on $\mathbb{K}$ given by
\begin{itemize}
	\item The coalgebra structure:
	$$
	\Delta(1)=\frac{1}{2} 1\otimes 1-\frac{1}{2} \mu\otimes \mu+\frac{i}{2} \nu\otimes \mu+\frac{i}{2} \mu\otimes \nu+\frac{1}{2} \nu\otimes \nu,
	$$
	\begin{eqnarray*}
		\Delta(g)&=&-\frac{1}{2} \mu\otimes 1+\frac{i}{2} \nu\otimes 1+b g\otimes g
		+b \mu \otimes g - ib \nu\otimes g\\
		&&-\frac{1}{2} 1\otimes \mu+b g\otimes \mu
		+b \mu \otimes \mu - ib \nu\otimes \mu\\
		&&+\frac{i}{2} 1\otimes \nu-ib g\otimes \nu
		-ib \mu \otimes \nu -b \nu\otimes \nu
	\end{eqnarray*}
	
	$$
	\Delta(\mu)=\frac{1}{2b}\mu\otimes\mu,
	$$	
	$$
	\Delta(\nu)=\frac{i}{2}\mu \otimes 1+\frac{1}{2}\nu\otimes 1
	+\frac{i}{2}1 \otimes \mu-\frac{i}{2b}\mu\otimes \mu+\frac{1}{2}1\otimes \nu,
	$$	
	$$
	\varepsilon(1)=2, \varepsilon(g)=\frac{1}{b},\varepsilon(\mu)=2b, \varepsilon(\nu)=-2ib,
	$$	
\item The antipode:
$$
S(1)=1,S(g)=(-1+\frac{1}{2b^{2}})\mu+i\nu,
$$

$$
S(\mu)=2b^{2}g+2b^{2}\mu-2i b^{2}\nu,
$$	
$$
S(\nu)=-2ib^{2}g-(-i+2ib^{2})\mu+ (1-2b^{2})\nu.
$$			
\end{itemize}

Next, we shall discuss the target algebra $\mathbb{K}_{t}$. By a straightforward computation, we have
$$
\varepsilon_{t}(\nu)=(b-i)\mu+(1+ib)\nu+ib, \varepsilon_{t}(\mu)=b1-bi\mu+b\nu,
$$
$$
\varepsilon_{t}(g)=\frac{1}{2b}1+(\frac{i}{2b}-1)\mu-(i+\frac{1}{2b})\nu.
$$
Set $Q=\mathbb{C}\textlangle1, (\nu-i \mu)\textrangle$.  Since
$$
\varepsilon_{t}(g)=\frac{1}{2b}-(i+\frac{1}{2b})(\nu-i \mu), \varepsilon_{t}(\mu)=b1+b(\nu-i \mu),
$$
$$
\varepsilon_{t}(\nu)=bi 1+(1+bi)(\nu-i \mu),
$$
it follows that $\varepsilon_{t}(\mathbb{K})\subseteq Q$. Since
$$
01+0\varepsilon_{t}(g)-i\varepsilon_{t}(\mu)+\varepsilon_{t}(\nu)=\nu-i \mu,
$$
we have $Q\subseteq \varepsilon_{t}(\mathbb{K})$. Thus
$$\varepsilon_{t}(\mathbb{K})= \mathbb{C}\textlangle1, (\nu-i \mu)\textrangle.$$
Next, we discuss the source algebra $\varepsilon_{s}(\mathbb{K})$. First, we have
$$
\varepsilon_{s}(g)=\frac{1}{2b}-(1+\frac{i}{2b})\mu+(\frac{1}{2b}-i)\nu, \varepsilon_{s}(\mu)=b+ib\mu-b\nu,
$$
$$
\varepsilon_{s}(\nu)=bi-(b+i)\mu+(1-bi)\nu.
$$
 Since
$$
\varepsilon_{s}(g)=\frac{1}{2b}+(\frac{1}{2b}-i)(\nu-i \mu), \varepsilon_{s}(\mu)=b 1-b(\nu-i \mu),
$$
$$
\varepsilon_{s}(\nu)=bi 1+(1-bi)(\nu-i \mu),
$$
it follows that $\varepsilon_{s}(\mathbb{K})\subseteq Q$. Since
$$
01+0\varepsilon_{s}(g)-i\varepsilon_{s}(\mu)+\varepsilon_{s}(\nu)=\nu-i \mu,
$$
we have $Q\subseteq \varepsilon_{s}(\mathbb{K})$. Thus
$$\varepsilon_{s}(\mathbb{K})= \mathbb{C}\textlangle1, (\nu-i \mu)\textrangle.$$
Through the discussion as above, we know that  $\varepsilon_{s}(\mathbb{K})= \varepsilon_{t}(\mathbb{K})$.
\begin{proposition}
	The subalgebra $\mathbb{C}\textlangle1, (\nu-i \mu)\textrangle$ of $\mathbb{K}$  is  separable algebras with the separable idempotent given by
	\begin{eqnarray*}
	q&=&\frac{1}{2}1\otimes 1-\frac{1}{2}\mu\otimes\mu-\frac{i}{2}\nu\otimes\mu
-\frac{i}{2}\mu\otimes\nu+\frac{1}{2}\nu\otimes\nu.
	\end{eqnarray*}
\end{proposition}
\begin{proof}
	The desired $q$ is
	\begin{eqnarray*}
		q&=&S(1_{(1)})\otimes 1_{(2)}\\
		&=&\frac{1}{2}  1-\frac{1}{2} S(\mu) \mu-\frac{i}{2} S(\nu) \mu-\frac{i}{2} S(\mu) \nu+\frac{1}{2} S(\nu) \nu\\
		&=&\frac{1}{2}  1-\frac{1}{2} (2b^{2}g+2b^{2}\mu+2i b^{2}\nu) \mu\\
		&&-\frac{i}{2} (2ib^{2}g-(i-2ib^{2})\mu+ (1-2b^{2})\nu) \mu-\frac{i}{2} (2b^{2}g+2b^{2}\mu+2i b^{2}\nu) \nu\\
		&&+\frac{1}{2} (2ib^{2}g-(i-2ib^{2})\mu+ (1-2b^{2})\nu) \nu\\
		&=&\frac{1}{2}1\otimes 1-\frac{1}{2}\mu\otimes\mu-\frac{i}{2}\nu\otimes\mu
		-\frac{i}{2}\mu\otimes\nu+\frac{1}{2}\nu\otimes\nu.
	\end{eqnarray*}
\end{proof}

\section{Integrals on $\mathbb{K}$ }
\label{xxsec0}
In this section, we shall discuss the left and right integrals in $\mathbb{K}$.

\subsection{ Left integrals in $\mathbb{K}$}
Set
$$
l=k_{1}1+k_{2}g+k_{3}\mu+k_{4}\nu.
$$
If $l$ is a left integral, then we have
\begin{equation}\label{ceq1}\tag{4.1}
xl=\varepsilon_{t}(x)l,\, x\in \{g,\mu,\nu\}.
\end{equation}
From (\ref{ceq1}), we have the following system of equations
\begin{equation}\label{ceq2}\tag{4.2}
2 b k_3+2 i b k_4-k_1-i k_2+k_4=0,
\end{equation}
\begin{equation}\label{ceq3}\tag{4.3}
k_1+k_4=\left(\frac{1}{b}+i\right) k_2,
\end{equation}
\begin{equation}\label{ceq4}\tag{4.4}
(2 b-i) k_1+(-1-2 i b) k_2+i \left(2 b k_3+k_4\right)=0,
\end{equation}
\begin{equation}\label{ceq5}\tag{4.5}
(1+2 i b) k_1+(2 b-i) k_2-2 b k_3-k_4=0,
\end{equation}
\begin{equation}\label{ceq7}\tag{4.6}
k_2=b \left(k_1-i k_2+k_4\right),
\end{equation}
\begin{equation}\label{ceq8}\tag{4.7}
(1+i b) k_1+b k_2+(-1-i b) k_4=2 b k_3,
\end{equation}
\begin{equation}\label{ceq9}\tag{4.8}
i b \left(k_2-2 k_3\right)+(b-i) k_4=(b-i) k_1.
\end{equation}

By solving the system of the equations (\ref{ceq2})-(\ref{ceq9}),  we have the following result.

\begin{theorem}
The left integral on $\mathbb{K}$ is 	$$
I^{L}(\mathbb{K})=\mathbb{C}<\left(\frac{1}{b}+i\right) \mu -\nu +1, -\frac{(1-2 b (b-i)) \mu }{2 b^2}+\frac{(1+i b) \nu }{b}+g>.
$$
\end{theorem}
\begin{proof}
	The solutions for the system of the equations (\ref{cqg1})-(\ref{cqg2}) are
	$$
	k_3=-\frac{k_2-2 b (b-i) \left(k_2+i k_1\right)}{2 b^2}, k_4=-k_1+\frac{(1+i b) k_2}{b}.
	$$
	Thus all the left  integrals are
	\begin{eqnarray*}
		&&k_{1}1+k_{2}g+(-\frac{k_2-2 b (b-i) \left(k_2+i k_1\right)}{2 b^2})\mu+(-k_1+\frac{(1+i b) k_2}{b})\nu\\
		&=&k_{1}(\left(\frac{1}{b}+i\right) \mu -\nu +1)+k_{2}(-\frac{(1-2 b (b-i)) \mu }{2 b^2}+\frac{(1+i b) \nu }{b}+g),
	\end{eqnarray*}
	so we get the desired result.
\end{proof}
\subsection{Right integrals in $\mathbb{K}$}
Set
$$
r=k_{1}1+k_{2}g+k_{3}\mu+k_{4}\nu.
$$
If $r$ is a right integral, then we have
\begin{equation}\label{ceq100}\tag{4.9}
	rx=r\varepsilon_{s}(x),\, x\in \{g,\mu,\nu\}.
\end{equation}
From (\ref{ceq100}), we have the following system of equations
\begin{equation}\label{cqg1}\tag{4.10}
-2 b \left(k_3+i k_4\right)+k_1-i k_2+k_4=0,
\end{equation}
\begin{equation}\tag{4.11}
	k_1+\left(-\frac{1}{b}+i\right) k_2=k_4,
\end{equation}
\begin{equation}\tag{4.12}
(2 b+i) k_1+i \left((2 b+i) k_2-2 b k_3+k_4\right)=0,
\end{equation}
\begin{equation}\tag{4.13}
	(1-2 i b) k_1+(2 b+i) k_2-2 b k_3+k_4=0,
\end{equation}
\begin{equation}\tag{4.14}
-i b k_1+b k_2-i b k_4+k_1+k_4=2 b k_3,
\end{equation}
\begin{equation}\tag{4.15}
b k_1+i (b+i) k_2=b k_4,
\end{equation}
\begin{equation}\label{cqg2}\tag{4.16}
(b+i) k_1+i b \left(k_2-2 k_3\right)+(b+i) k_4=0.
\end{equation}

By solving the system of the equations (\ref{cqg1})-(\ref{cqg2}), we have the following result.

\begin{theorem}
	The right integral space of  $\mathbb{K}$ is
	$$
I^{R}(\mathbb{K})=\mathbb{C}<\left(\frac{1}{b}-i\right) \mu +\nu +1, \frac{(-1+2 b (b+i))  }{2 b^2}\mu+\frac{  (\text{ib}-1)}{b}\nu+g>.
	$$
\begin{proof}
The solutions for the system of the equations (\ref{cqg1})-(\ref{cqg2}) are
$$
k_{3}=-\frac{k_2+2 i b (b+i) \left(k_1+i k_2\right)}{2 b^2}, k_{4}=k_1+\frac{i (b+i) k_2}{b}.
$$
Thus all the right integrals are
\begin{eqnarray*}
&&k_{1}1+k_{2}g+(-\frac{k_2+2 i b (b+i) \left(k_1+i k_2\right)}{2 b^2})\mu+(k_1+\frac{i (b+i) k_2}{b})\nu\\
&=&k_{1}(\left(\frac{1}{b}-i\right) \mu +\nu +1)+k_{2}( \frac{(-1+2 b (b+i))  }{2 b^2}\mu+\frac{  (i b-1)}{b}\nu+g),
\end{eqnarray*}
so we get the desired result.
\end{proof}
\end{theorem}

\noindent{\bf Data availability}  Data sharing not applicable to this article as no datasets were
generated or analysed during the current study.

\noindent{\bf Declarations}

\noindent {\bf Conflict of interest}  The authors have no relevant financial or non-financial
interests to disclose.


\begin{thebibliography}{10}

\bibitem{Bohm} Böhm, G., Nill, F., Szlachányi, K. (1999). Weak Hopf algebras I. Integral theory and C-structure. J. Algebra 221:
385–438.
\bibitem{Pierce} R. S. Pierce, Associative Algebras, (Graduate Texts in Matematics 88) Springer-Verlag
1982
\bibitem{Etingof}P. Etingof, O. Schiﬀmann, Lectures on the dynamical Yang-Baxter equations, math.QA/9908064, 1999.
\bibitem{Kadison}L. Kadison, D. Nikshych, Frobenius extensions and weak Hopf algebras, Journal of Algebra, 2001, 244(1), 312-342.
\bibitem{ozdemir} \"{O}zdemir M.,  Introduction to Hybrid numbers, Adv. Appl. Clifford Algebras, 2018, 28(1), 11.
%\bibitem{Ding2}
%D. G. Wang, J. J. Zhang, G. Zhang,  Lower bounds of growth of Hopf algebras.  \emph{Trans. Amer. Math. Soc.}, %\textbf{365}(2013), 4963--4986.

\bibitem{ozdemir} S. Montgomery, Hopf Algebras and their Actions on Rings,  ICBMS, Chicago, 1993.





\end{thebibliography}
\end{document}